\documentclass[preprint,5p,authoryear]{elsarticle}
\usepackage{indentfirst}
\usepackage{amsmath,amssymb}
\usepackage{graphicx,subfigure,epstopdf}
\usepackage{epic}
\usepackage{subfigure}
\usepackage{graphicx}
\usepackage{epsfig}
\usepackage{arydshln}
\usepackage{enumerate}
\usepackage[english]{babel}
\usepackage{fancyhdr}
\usepackage{lastpage}
\usepackage{color}
\usepackage{dsfont}
\usepackage{tikz}
\usetikzlibrary{positioning}
\usetikzlibrary{shapes,arrows}
\usepackage{caption}
\usepackage{changes}
\usepackage{booktabs}
\usepackage{threeparttable}
\usepackage{animate}
\usepackage{natbib}
\usepackage{caption}
\usepackage[hyperfootnotes=false]{hyperref}
\usepackage{todonotes}
\hypersetup{
colorlinks=true,
linkcolor=black,
citecolor=cyan
}
\usetikzlibrary{positioning}

\setcitestyle{longnamesfirst,round,authoryear,semicolon,sort&compress}
\newcommand{\EQ}{\begin{eqnarray}}
\newcommand{\EN}{\end{eqnarray}}
\newcommand{\EQQ}{\begin{eqnarray*}}
\newcommand{\ENN}{\end{eqnarray*}}
\newcommand{\col}{\mbox{col}}
\newcommand{\diag}{\mbox{diag}}

\newtheorem{thm}{Theorem}

\newtheorem{lem}{Lemma}
\newtheorem{rem}{Remark}
\newtheorem{defi}{Definition}
\newtheorem{prop}{\it Property}

\newtheorem{ass}{Assumption}

\newenvironment{proof}{\noindent{\em Proof:\/}}{\hfill $\Box$\par}

\newcommand{\myr}{}

\graphicspath{{Figures/}}

\allowdisplaybreaks
\begin{document}
\begin{sloppypar}
\begin{frontmatter}

\title{Extremum Seeking Nonlinear Regulator with Concurrent Uncertainties in Exosystems and Control Directions}

\tnotetext[t1]{This work was supported in part by NSERC and in part by NSFC under grant No. 62073168. This paper has not been presented at any conference.}%
 
\author[Wang]{Shimin Wang}\ead{bellewsm@mit.edu}
\author[Martin]{Martin Guay}\ead{guaym@queensu.ca}
\author[Denis]{Denis Dochain}\ead{denis.dochain@uclouvain.be}

\address[Wang]{Massachusetts Institute of Technology, Cambridge,
MA 02142, USA}
\address[Martin]{Department of Chemical Engineering, Queen's University, Kingston, Ontario, K7L 3N6, Canada}
\address[Denis]{ICTEAM, UCLouvain, Bâtiment Euler, avenue Georges lemaître 4-6, 1348 Louvain-la-Neuve, Belgium}

\begin{abstract}This paper proposes a non-adaptive control solution framework to the practical output regulation problem (PORP) for a class of nonlinear systems with uncertain parameters, unknown control directions and uncertain exosystem dynamics. 
The concurrence of the unknown control directions and uncertainties in both the system dynamics and the exosystem pose a significant challenge to the problem. 
In light of a nonlinear internal model approach, we first convert the robust PORP into a robust non-adaptive stabilization problem for the augmented system with integral Input-to-State Stable (iISS) inverse dynamics.
By employing an extremum-seeking control (ESC) approach, the construction of our solution method avoids the use of Nussbaum-type gain techniques to address the robust PORP subject to unknown control directions with time-varying coefficients.
The stability of the non-adaptive output regulation design is proven via a Lie bracket averaging technique where uniform ultimate boundedness of the closed-loop signals is guaranteed.
{\myr As a result, the practical output regulation problem can be solved using the proposed non-adaptive and non-Nussbaum-type framework.
Moreover, both the estimation and tracking errors uniformly asymptotically converge to zero, provided that the frequency of the dither signal goes to infinity.}
Finally, a simulation example with unknown coefficients is provided to exemplify the validity of the proposed control solution frameworks.
\end{abstract}
\end{frontmatter}

\section{Introduction}
The control problem of output regulation or servomechanism aims at achieving asymptotic tracking of reference signals while rejecting the steady-state effect of disturbances \citep*{isidori1990output,isidori2003robust,huang2004nonlinear,bin2020model,wang2022robust}.
In particular, a comprehensive framework in \cite{marconi2008uniform} has been presented for asymptotically solving the PORP.
As in \cite{marconi2008uniform,liu2009parameter}, for instance, most existing studies require knowledge of the control direction {\it a priori}. 
The control directions naturally play an essential role in solving output regulation problems for both linear and nonlinear systems \citep*{liu2008global}. 
A wrong control direction can force the output regulation error of the feedback control systems to drift away from the desired control objective \citep*{chen2019nussbaum}.
Furthermore, many practical applications have stimulated the investigation of output regulation problems subject to unknown control directions with time-varying coefficients, such as the autopilot design of unmanned autonomous surface vessels in \cite{wang2015prescribed}. 
The formation and station-keeping control of multiple networked autonomous high-altitude balloons use stratospheric wind currents as propulsion to move forward and navigate; however, the stratospheric wind currents are unknown, time-varying, and unpredictable \citep*{vandermeulen2017distributed}.
In a recent study, \cite{dibo2023extremum}, the requirement for the knowledge of the Hessian sign information for the design of an extremum-seeking controller was alleviated using a switching monitoring function-based scheme.

The main objective of the present study focuses on a class of output-feedback uncertain nonlinear systems subject to unknown control directions in the following form, previously investigated in \cite{liu2009parameter}:
\begin{subequations}\label{multirigda}
\begin{align}
\dot{z}&=F(w)z+G(y,v,w)y+D_1(v,w),\\
\dot{y}&=H(w)z+K(y,v,w)y+b(w,v)u+D_{2}(v,w),\\
e&=y-q(v,w),\nonumber
\end{align}\end{subequations}
where $\col(z,y)\in \mathds{R}^{n_{s}}$ is the whole system state, $y\in \mathds{R}$ is the system output, $u\in\mathds{R}$ is the control input, $w\in \mathds{R}^{n_w}$ collects the uncertain parameters or parametric uncertainties, $b(w,v)$ is continuous in its arguments, satisfying
\begin{equation*}
{\myr b(w,v)^{2} }>0,~~ \forall w\in \mathds{R}^{n_w},
\end{equation*}
$e\in \mathds{R}$ is the error output, and $v\in \mathds{R}^{n_v}$ is the exogenous signal representing the reference input to be tracked or disturbance to be rejected. The matrix $F(w)$ is Hurwitz for all $w\in \mathds{R}^{n_w}$. The signal $v$ is assumed to be generated by the following uncertain exosystem:
\begin{align}\label{leader}
  \dot{v}=&S(\sigma)v,
\end{align}
where $\sigma\in \mathds{S}\subset \mathds{R}^{n_{\sigma}}$ represents the unknown parameters, and $q(v,w)\in \mathds{R}$ is the output of the exosystem. We assume that all the functions in system \eqref{multirigda} are sufficiently smooth satisfying $$D_1(0,w)=0,\ D_2(0,w)=0\ \textnormal{and}\ q(0,w)=0,\ \forall w\in \mathds{R}^{n_w}.$$

Multiple versions of the output regulation problem for various nonlinear system dynamics subject to unknown control directions and a known exosystem have been extensively researched for over a decade \citep*{liu2006global, ding2015adaptive}.
The output regulation problem is challenging to address satisfactorily, when the dynamics of the control system are subject to unknown control directions and an uncertain exosystem. 
For example, the robust output regulation problem over unknown control directions mixed with an uncertain exosystem for nonlinear system dynamics in lower triangular forms has been addressed in \cite{guo2017global} using  Nussbaum function-based techniques.
The output regulation problem without a known control direction has stimulated significant research interest in the control community \cite{liu2006global,oliveira2011output}. It remains a relevant and challenging research topic as outlined in \cite{liu2017cooperative,zhang2023lie,hua2023adaptive,Rovithakis2024}.
This paper proposes a non-Nussbaum function-based control solution framework for the Robust PORP subject to unknown control directions described as follows:

\emph{Given system \eqref{multirigda}, \eqref{leader} with compact subsets $\mathds{V}\in\mathds{R}^{n_v}$ and $\mathds{W}\in \mathds{R}^{n_w}$, for any constant $\nu>0$, find a non-adaptive and non-Nussbaum function-based control law such that for all initial conditions with $v(0)\in \mathds{V}$ and $w\in \mathds{W}$, $\lim\limits_{t\rightarrow\infty}|e(t)|\leq \nu$ independent of the unknown control directions with time-varying coefficients}.

The Nussbaum function-based technique, initially proposed in \cite{nussbaum1983some}, has been widely considered in multiple studies including \cite{liu2006global,guo2016cooperative} to handle the unknown control directions. 
While they have successfully solved difficult control problems, these techniques can suffer from poor transient performance, as pointed out in \cite{scheinker2012minimum}.
{\myr Nussbaum functions have often been considered as the preferred solutions for unknown control direction problems}. 
In fact, the problems over unknown time-varying control coefficients can only be addressed using some particular Nussbaum functions as shown in \cite{liu2008global}. 
Moreover, the Nussbaum gain approach was used in \cite{bechlioulis2011robust} to investigate the robust prescribed performance control of $n$th order cascade nonlinear system with partial-state feedback.
It is important to note that Nussbaum function-based design techniques fail to achieve exponential stability even in the absence of uncertainties.
In addition, the overshoot phenomenon can be observed in almost every paper, such as \cite{liu2006global,liu2014adaptive}.
Furthermore, a counter-example was proposed in \cite{chen2019nussbaum} to show that the existing Nussbaum
functions are not always effective in multi-variable and/or time-varying control coefficients with unknown
signs.
%

Recent results investigated by \cite{scheinker2012minimum} have shown that the extremum-seeking algorithm can also be applied to solve the semi-global stabilization of unstable and time-varying systems with unknown time-varying control directions and full state feedback.
{\myr \cite{zhang2023lie} extended these results to the stabilization of linear uncertain systems with unknown control directions, using a bounded extremum-seeking controller to account for time-varying delays caused by delayed state measurements.}
Extremum-seeking control has a long history. 
The recent comprehensive survey  \cite{scheinker2024100} provides a complete account of the field over the last 100 years.
This technique aims to steer an unknown dynamical system to the optimum of a partially or completely unknown map \citep*{dehaan2005extremum,tan2006non,krstic2000stability,yang2022periodic}.
Particularly, \cite{guay2019extremum,guay2019extremumB} generalized an extremum-seeking control approach to solving the output regulation of a nonlinear control system using a post-processing framework.
In this study, we deal with the robust PROP of output feedback systems with an unknown control direction mixed with an uncertain exosystem.
{\myr Moreover, by employing a nonlinear internal model-based approach, our paper transforms the robust PROP into a robust non-adaptive stabilization problem for a class of nonlinear system dynamics in output feedback form with iISS inverse dynamics. 
This framework includes the full state feedback control system case described in \cite{dehaan2005extremum} as a special instance.}

By employing the extremum-seeking control approach in \cite{guay2019extremum,guay2019extremumB}, we will construct control laws that avoid the use of Nussbaum-type gain techniques and solve the robust PROP subject to unknown control directions with time-varying coefficients.
The stability of the non-adaptive output regulation design is proven via
a Lie bracket averaging technique \citep*{durr2013lie} where uniform ultimate bounded signals produced within the closed-loop system can be guaranteed.
{\myr As a result, the practical output regulation problem can be addressed by the proposed non-adaptive and extremum-seeking control approach. This further implies that the tracking error can uniformly asymptotically converge to a compact set determined by the frequency of the dither signal.
Moreover, both the output regulation and parameter estimation errors will converge to zero exponentially as time approaches infinity, provided that the frequency of the dither signal tends to infinity.
Clearly, the results enhance and differ from the results in \cite{liu2006global,guo2016cooperative}}.
Finally, a numerical example for a class of output feedback control nonlinear systems with
an unknown time-varying coefficient is provided to demonstrate the
effectiveness of the proposed non-Nussbaum-based control solution framework.


The rest of this paper is organized as follows. In Section \ref{section1}, some standard assumptions are introduced, and a non-adaptive output regulation design is given. One lemma is established, followed by the presentation of some existing results from \cite{durr2013lie,chen2015stabilization}. The main result is presented in Section \ref{section4}. A numerical example is provided in Section \ref{section5} to illustrate the proposed design.

\textbf{Notation:} $\|\cdot\|$ is the Euclidean norm. $\emph{Id} : \mathds{R}\rightarrow \mathds{R}$ is the identity function. For $X_1,\dots,X_N\in \mathds{R}^{n}$, let $\col(X_1,\dots,X_N)=[X_{1}^{\!\top},\dots ,X^{\!\top}_{N}]^{\!\top}$. For two vector fields, $\alpha_i(x)$ and $\alpha_j(x)$, the Lie bracket denoted by $[\alpha_i(x), \alpha_j(x)]$ is given by: $$\left[\alpha_i(x), \alpha_j(x)\right]=\frac{\partial \alpha_j}{\partial x}\alpha_i(x)-\frac{\partial \alpha_i}{\partial x}\alpha_j(x).$$
A function $\alpha: \mathds{R}_{\geq 0}\rightarrow \mathds{R}_{\geq 0}$ is of class $\mathcal{K}$ if it is continuous, positive definite, and strictly increasing. $\mathcal{K}_o$ and $\mathcal{K}_{\infty}$ are the subclasses of bounded and unbounded $\mathcal{K}$ functions, respectively. For functions $f_1(\cdot)$ and $f_2(\cdot)$ with compatible dimensions, their composition $f_{1}\left(f_2(\cdot)\right)$ is denoted by $f_1\circ f_2(\cdot)$. For two continuous and positive definite functions $\kappa_1(s)$ and $\kappa_2(s)$, $\kappa_1\in \mathcal{O}(\kappa_2)$ means $\limsup\limits_{s\rightarrow 0^{+}}\frac{\kappa_1(s)}{\kappa_2(s)}<\infty$.

\section{Preliminaries}\label{section1}%

\subsection{Standard Assumptions}\label{section1a}%

In the section, we list several assumptions required in the analysis of the proposed approach.
\begin{ass}\label{ass1} All the eigenvalues of $S(\sigma)$ are distinct with zero real part, for all $\sigma\in \mathds{S}$.
\end{ass}
Assumption \ref{ass1} is such that the general solution of \eqref{leader} is a sum
of finitely many sinusoidal functions with frequencies
depending on the eigenvalues of $S(\sigma)$ and amplitudes and phase
angles depending on the initial condition.
\begin{ass}\label{ass2} 
The system \eqref{multirigda} under investigation is minimum-phase, i.e., $F(w)$ is Hurwitz for all $w\in\mathds{W}$. Moreover, there are smooth nonlinear functions $\textbf{z}(v, \sigma, w)$ with $\textbf{z}(0,0,0)=0$ such that for any $v\in \mathds{R}^{n_v}$, $\sigma\in \mathds{S}$ and $w\in \mathds{W}$:
\begin{align*}
  \frac{\partial \textbf{z}(v,\sigma,w)}{\partial v} S(\sigma)v&=F(w)\textbf{z}(v,\sigma,w)\notag\\
  &+G\left( q(v,w),v,w \right) q(v,w)+D_1(v,w).
\end{align*}
\end{ass}
The above assumptions ensure the following useful condition: the solution of the exosystem \eqref{leader} can be expressed as the finite sum of sinusoidal functions. {\myr Clearly}, there exists a compact set $\mathds{V}$, such that for any $v(0)\in \mathds{V}$, $v(t)\in \mathds{V}$ for all $t\geq0$.
Under Assumption \ref{ass2}, let $\textbf{y}(v,w)=q(v,w)$ and
\begin{align*}
\textbf{u}(v, \sigma, w)=&\ {\myr b(v, w)^{-1}}\left(\frac{\partial \textbf{y}(v,w)}{\partial v}S(\sigma)v-H(w)\textbf{z}\left(v, \sigma, w\right)\right.\\
& -K(\textbf{y}(v,w),v,w)\textbf{y}(v,w)-D_{2}(v,w)\bigg).
\end{align*}
We can verify that $\textbf{z}(v,\sigma,w)$, $\textbf{y}(v,w)$ and $\textbf{u}(v, \sigma, w)$ are the solutions of the regulator equations associated with systems (\ref{multirigda}) and (\ref{leader}) \citep*{liu2009parameter}.

\begin{ass}\label{ass4} The functions $\textbf{u}(v,\sigma,w)$ are polynomials in $v$ with coefficients depending on $w$ and $\sigma$ for all $\textnormal{\col}(w,\sigma)\in \mathds{W}\times \mathds{S}$.
\end{ass}
\begin{rem}
Moreover, from \cite{huang2001remarks} and \cite{liu2009parameter}, under Assumptions  \ref{ass1} and \ref{ass4}, for the function $\textbf{u}(v,\sigma,w)$, there is an integer $s^*>0$ such that $\textbf{u}(v,\sigma,w)$ can be expressed by
$$\textbf{u}(v,\sigma,w)=\sum_{i=1}^{s^*}C_{i}(v(0), w,\sigma)e^{\imath \hat{\omega}_it},$$
for some functions $C_{i}(v(0), w,\sigma)$, where $\imath$ is the imaginary unit and $\hat{\omega}_i$ are distinct real numbers for $0\leq i \leq s^*$.
\end{rem}
To guarantee that the steady-state
input signal $\mathbf{u}(\mu)$ is sufficiently rich of order $n$ ($n\in\{2s^*, 2s^*-1\}$), the following assumption is needed.
\begin{ass}\label{ass3} For any $v(0)\in \mathds{V}$, $w\in \mathds{W}$ and $\sigma\in \mathds{S}$, $C_{i}(v(0), w,\sigma)\neq 0$ for all $i\in\{1,\dots,s^*\}$.
\end{ass}
 
\subsection{Nonlinear Internal Model Design}
As shown in \cite{huang2004nonlinear}, under Assumptions \ref{ass1} and \ref{ass2}, there exist positive integers $n$, such that $\textbf{u}(\mu)$ satisfy, for all $\mu\in \mathds{V}\times\mathds{W}\times \mathds{S}$ with {\myr $\mu=\col(v, \sigma, w)$},
\begin{align} \label{aode} 
\frac{d^{n}\textbf{u}(v(t), \sigma, w)}{dt^{n}}&+a_{1}(\sigma)\textbf{u}(v(t), \sigma, w)\notag\\
&+\dots+a_{n}(\sigma)\frac{d^{n-1}\textbf{u}(v(t), \sigma, w)}{dt^{n-1}}=0,
\end{align}
where $a_{1}(\sigma), a_{2}(\sigma),\dots$, and $ a_{n}(\sigma)$ are all belong to $\mathds{R}$. Under Assumptions \ref{ass1} and \ref{ass2}, equation \eqref{aode} satisfies that the following polynomial 
$$P(s)=s^{n}+a_{1}(\sigma)+a_{2}(\sigma)s+\dots+a_{n}(\sigma)s^{n-1}$$ 
contains distinct roots with zeros real parts for all $\sigma\in \mathds{S}$. 
Let $a(\sigma)=\col(a_{1}(\sigma), \dots, a_{n}(\sigma))$ 
and $${\bf{\xi}}\left(\mu \right)= \col\left(\textbf{u}\left(\mu \right),\frac{d\textbf{u}\left(\mu \right)}{dt},\dots,\frac{d^{n-1}\textbf{u}\left(\mu \right)}{dt^{n-1}}\right).$$ 
In  addition, we define
\begin{align*}
  \Phi(a) =&\ \left[
                      \begin{array}{@{}c|c@{}}
                       \textbf{0}_{(n-1)\times 1} & I_{n-1} \\
                        \hline
                        -a_{1} &-a_{2},\dots,-a_{n} \\
                      \end{array}
                    \right], \;\Gamma^{\!\top} = \left[
                \begin{array}{@{}c@{}}
                  1 \\
                 \textbf{0}_{n-1}
                \end{array}
              \right]_{1\times n}.
\end{align*}
The expressions ${\bf\xi}\left(v,\sigma,w\right)$, $\Phi$ and $\Gamma$ satisfy the following so-called steady-state generator with output $u$
\begin{align}
 \frac{ \partial {\bf\xi}\left(\mu\right)}{\partial v}S(\sigma)v=&\ \Phi(a)  {\bf \xi}\left(\mu\right),\nonumber\\
   \textbf{u}\left(\mu\right)= &\ \Gamma  {\bf \xi}\left(\mu\right),\label{xi-exosystem}
\end{align}
which can be used to produce the steady-state input signal $\textbf{u}(\mu)$.
Next, we define a dynamic compensator given by:
\begin{subequations}\label{gecopen}\begin{align}
  \dot{\eta}=&\ M\eta+N\pi,\\
  \dot{\pi}=&-\pi+u,\\
    \dot{\vartheta}=&- \Theta\eta[\eta^{\!\top}\vartheta-\pi],
\end{align}\end{subequations}
where $\Theta$ is any positive constant, $\eta\in \mathds{R}^{n}$, $\vartheta\in \mathds{R}^{n}$, $\pi \in\mathds{R}$,
\begin{align*}
 M=&\ \left[
                      \begin{array}{@{}c|c@{}}
                       \textbf{0}_{{(n-1)\myr}\times 1} & I_{n-1} \\
                        \hline
                        -m_{1} &-m_{2},\dots,-m_{n} \\
                      \end{array}
                    \right],\;\;N=\col(0,\dots,0,1).
\end{align*}
Following Theorem 3.1 in \cite{xu2017constructive}, we perform the following transformation
\begin{align*}
\theta\left(\mu\right)=& T(a){\bf \xi}\left(\mu\right),& 
\varrho =m-a,&  &
\varpi \left(\mu\right)=\varrho ^{\!\top} T(a){\bf\xi}\left(\mu\right),
\end{align*}
where $m=\col(m_1,\dots,m_n)$ and $T(a)$ is defined as
\begin{align}\label{Tmatrix}
 {\myr T(a)^{-1}}=\left[
       \begin{array}{@{}c@{}}
          \varrho ^{\!\top}\left[\Phi(\sigma)+I_n\right] \\
         \varrho ^{\!\top}\left[\Phi(\sigma)+I_n\right]\Phi(\sigma)\\
         \vdots \\
          \varrho ^{\!\top}\left[\Phi(\sigma)+I_n\right]\Phi(\sigma)^{n-1} \\
       \end{array}
     \right].
 \end{align}
 It is noted that the pair $\big(\Phi(\sigma), \left[m-a\right]^{\!\top}\big)$ is observable and all the eigenvalues of $\Phi(\sigma)$ have zeros real parts. Hence, the matrix $T(a)$ is well-defined.
As shown in \cite{xu2016output} using the Cayley–Hamilton theorem for $\Phi(a)$ and using $ {\myr T(a)^{-1}}T(a)=I$, the matrix $T(a)$ defined by \eqref{Tmatrix} satisfies a multiplicative commutative property such that
$$\Phi(a)=T(a)\Phi(a) {\myr T(a)^{-1}}.$$
Then, it can be verified that the $M$, $N$, $\Phi(a)$ and $T(a)$ satisfy the following matrix equation $$T(a)\Phi(a)-MT(a)=N\varrho^{\!\top}T(a).$$
We also have the following equations satisfying
\begin{subequations}\label{nonIM}
\begin{align}
\dot{\theta}\left(\mu\right)
&= T(a)\Phi(a) {\myr T(a)^{-1}}\theta\left(\mu\right)\notag\\
&=M\theta\left( \mu\right)+N\left[m-a\right]^{\!\top}\theta\left(\mu\right)\notag\\
&=M\theta\left(\mu\right)+N\varpi\left(\mu\right),\\
 \textbf{u}\left(\mu\right)&=\Gamma  {\myr T(a)^{-1}}\theta\left(\mu\right)\notag\\
& = \varrho ^{\!\top}\left[\Phi(a)+I_n\right]\theta\left(\mu\right)\notag\\
 & = \varrho^{\!\top}\left[\Phi(m-\varrho)+I_n\right]\theta\left(\mu\right) 
\textbf{ =:}\chi\left(\theta\left(\mu\right), \varrho\right),\label{nonIM2}\\
\textbf{0}&=  \theta\left(\mu\right)\left[{\myr \theta\left(\mu\right)^{\!\top}}\varrho-\varpi \left(\mu\right)\right].
\end{align}
\end{subequations}
Equation \eqref{nonIM} is called a nonlinear internal model of the system (\ref{multirigda}) (see \cite{xu2016output}). Motivated by the proposed framework in \cite{huang2004general,xu2016output}, the Robust PORP \eqref{aumet} of system \eqref{multirigda} is equivalent to a stabilization problem of a well-defined augmented system via the proposed nonlinear internal model approach. To achieve our goal, we first establish the following nonlinear functions for the signals $\eta$ and $\vartheta$ to provide the estimation of $\chi\left(\theta\left(\mu\right),\varrho\right)$ in \eqref{nonIM2}.

From Assumption \ref{ass1}, it follows that $\theta\left(\mu\right)$ and $\varpi \left(\mu\right)$ belongs to some compact set $\mathds{D}$.
 To construct the augmented system, we define a smooth function $\chi_s: \mathds{R}^{2n}\mapsto \mathds{R}$ such that
\begin{align}\label{chisatu}
\chi_s(\eta, \vartheta)=\left\{ \begin{array}{cc}
                                   \chi(\eta, \vartheta),& \textnormal{if}~(\eta,\vartheta)\in \mathds{D};\\
                                   0,& \textnormal{if}~(\eta,\vartheta)\notin \mathds{B}.\\
                                 \end{array}
\right.
\end{align}
In the above, a specific design can be 
\begin{align*}
\chi_s(\eta, \vartheta)=\chi(\eta, \vartheta)\Psi(\delta+1-\|\col(\eta, \vartheta)\|^2),
\end{align*}
with a compact support where $\Psi(s)=\frac{\psi(s)}{\psi(s)+\psi(1-s)}$, $\delta=\max_{(\eta,\vartheta)\in \mathds{D}}\|\col(\eta, \vartheta)\|^2$, $\mathds{B}=\{(\eta,\vartheta)| \|\col(\eta, \vartheta)\|^2\leq \delta+1\}$ and 
\begin{align*}
\psi(s)=\left\{ \begin{array}{cc}
                    e^{-s^{-1}}, & \textnormal{if}~ s>0;\\
                    0, & \textnormal{if}~s\leq 0.
                \end{array}
\right.
\end{align*}
To facilitate the design, we perform the following coordinate transformation,
 \begin{align*}
       \bar{z}  &=z-\textbf{z}(\mu),\quad \bar{\eta} =\eta-\theta\left(\mu\right),\quad \bar{u}= u-\chi_s(\eta, \vartheta),\\
       \bar{\vartheta} &=\vartheta-\varrho,\quad\bar{\pi}=\pi-\varpi \left(\mu\right)- {\myr b\left(\mu\right)^{-1}}e.
 \end{align*}
{\myr Applying this transformation to systems \eqref{multirigda} and the exosystem \eqref{leader}, along with the non-adaptive internal model \eqref{gecopen}, we obtain the following augmented system:}
\begin{subequations}\label{aumet}\begin{align}
  \dot{\bar{z}}& = F\left(w\right)\bar{z}+\bar{G}(e,w)e, \label{aumet-eq1}\\
    \dot{\bar{\eta}}&= M  \bar{\eta} + \bar{\varepsilon}(\bar{\pi}, e),\label{aumet-eq2}\\
       \dot{\bar{\pi}} &=-\bar{\pi}-\bar{\delta}(\bar{z}, e,\mu),\label{aumet-eq3}\\
    \dot{\bar{\vartheta}} &=-\Theta\theta\left(\mu\right){\myr \theta\left(\mu\right)^{\!\top}}\bar{\vartheta}+ \bar{\gamma}(\bar{\eta}, \bar{\pi}, e, \mu),\label{aumet-eq4}\\
  \dot{e}&=\bar{g}(\bar{z}, e, \bar{\eta}, \bar{\vartheta},\mu)+b(\mu)\bar{u},\label{aumet-eq5}
\end{align}\end{subequations}
where $\bar{\varepsilon}(\bar{\pi}, e)=N\bar{\pi}+N {\myr b\left(\mu\right)^{-1}}e$,
\begin{align}
 \bar{G}\left(e,\mu\right)=&\ G(q(\mu)+e)(q(\mu)+e)-G(q(\mu),\mu)q(\mu),\nonumber\\
 \bar{\gamma}(\bar{\eta},\bar{\pi}, e,\mu)=&\ \Theta\bar{\eta}\left[\bar{\pi}+\varpi \left( \mu\right)+ {\myr b\left(\mu\right)^{-1}}e\right],\nonumber\\
 &+\Theta\theta\left(\mu\right)\left[\bar{\pi}+ {\myr b\left(\mu\right)^{-1}}e\right]\nonumber\\
 &-\Theta\theta\left(\mu\right)\bar{\eta}^{\!\top}\bar{\vartheta}-\bar{\eta}[\bar{\eta}+\theta\left( \mu\right)]^{\!\top}\bar{\vartheta}\nonumber\\
 &-\Theta\bar{\eta}[\bar{\eta}+\theta\left(\mu\right)]^{\!\top}\varrho-\theta\left(\mu\right)\bar{\eta}^{\!\top}\varrho,\nonumber\\
\bar{\delta}(\bar{z}, e,\mu)=&\ {\myr b\left(\mu\right)^{-1}}\left[e+H(w)\bar{z}+ \bar{K}\left(e,\mu\right)\right]\nonumber\\
&-\frac{\partial  {\myr b\left(\mu\right)^{-1}}}{\partial v}S(\sigma)ve,\nonumber\\
 \bar{K}\left(e,\mu\right)=&\ K(q(\mu)+e)(q(\mu)+e)-K(q(\mu),\mu)q(\mu),\nonumber\\
 \bar{\chi}(\bar{\eta}, \bar{\vartheta})=&\ \chi_s(\bar{\eta}+\theta\left(\mu\right), \bar{\vartheta}+\varrho)-\chi_s\left(\theta\left(\mu\right), \varrho\right),\nonumber\\
\bar{g}(\bar{z}, e, \bar{\eta}, \bar{\vartheta},\mu) =&\ H(w)\bar{z}+ \bar{K}\left(e,\mu\right)+b(\mu)\bar{\chi}(\bar{\eta}, \bar{\vartheta}).\nonumber
\end{align}
It is important to note that \eqref{aumet} is the augmented system which is equivalent to the original plant \eqref{multirigda} and the internal model \eqref{gecopen}. We can also show that
 \begin{align*}
 \bar{G}\left(0,\mu\right)=&\ 0,& \varepsilon(0, 0)=0,& &\bar{\delta}(0, 0,\mu)=0,\\
 \bar{K}\left(0,\mu\right)=&\ 0,& \bar{\chi}(0, 0)=0,& &\bar{\gamma}(0,0,0,\mu)=0.
 \end{align*}
As a result of this transformation, we see that the Robust PORP of \eqref{aumet} can be addressed by the stabilization of the augmented system \eqref{aumet}. The stabilization problem is solved using the following lemmas
where we introduce the following key properties of system \eqref{aumet} under Assumptions \ref{ass1}, \ref{ass2}, \ref{ass4} and \ref{ass3}. 
\begin{lem}\label{lemmabodev} %
For the system \eqref{aumet} under Assumptions \ref{ass1}, \ref{ass2}, \ref{ass4} and \ref{ass3}, we have the following properties:
\begin{prop}\label{property1}%
There are smooth integral input-to-state Lyapunov functions $V_0:=V_0(\bar{z})$, $V_{1}:=V_{1}(\bar{\xi})$, and $V_2:=V_2\big(t, \bar{\vartheta}\big)$ satisfying
\begin{subequations}
\begin{align}
\underline{\alpha}_0\|\bar{z}\|^2&\leq V_0(\bar{z})\leq \bar{\alpha}_0\|\bar{z}\|^2,\notag\\
\dot{V}_0\big|_{\eqref{aumet-eq1}} &\leq  -\alpha_0 V_0+\delta_0(e^2),\label{V0}\\
\underline{\alpha}_1\|\bar{\xi}\|^2&\leq V_1(\bar{\xi})\leq \bar{\alpha}_1\|\bar{\xi}\|^2,\notag\\
\dot{V}_1\big|_{\eqref{aumet-eq2}+\eqref{aumet-eq3}} &\leq  -\alpha_1 V_1+\beta_1V_0+\delta_1(e^2),\label{V1}\\
\underline{\alpha}_2(\|\bar{\vartheta}\|^2)&\leq V_2(\bar{\vartheta})\leq \bar{\alpha}_2(\|\bar{\vartheta}\|^2),\notag\\
\dot{V}_2\big|_{\eqref{aumet-eq4}} &\leq  -\alpha_2(V_2) +\delta_3(V_1)+\delta_2(e^2),\label{V2}
\end{align}
\end{subequations}
for positive constants $\underline{\alpha}_0$, $\bar{\alpha}_0$, $\underline{\alpha}_1$, $\bar{\alpha}_1$, $\alpha_0$, $\alpha_1$, and $\beta_1$, and comparison functions $\underline{\alpha}_2(\cdot)\in \mathcal{K}_{\infty}$, $\bar{\alpha}_2(\cdot)\in \mathcal{K}_{\infty}$, $\alpha_2(\cdot)\in \mathcal{K}_{o}$, $\delta_0(\cdot)\in \mathcal{K}$, $\delta_1(\cdot)\in \mathcal{K}$, $\delta_2(\cdot)\in \mathcal{K}_{\infty}$ and $\delta_3(\cdot)\in \mathcal{K}_{\infty}$ with $\bar{\xi}=\textnormal{\col}(\bar{\eta},\bar{\pi})$.
\end{prop}
\begin{prop}\label{property2} There are positive constants $\phi_0$, $\phi_1$, $\phi_2$ and a smooth function $\delta_5(\cdot)\in \mathcal{K}$ such that
$$\|\bar{g}(\bar{z}, e, \bar{\eta}, \bar{\vartheta},\mu)\|^2\leq \phi_0 V_0 +\phi_1V_1+\phi_2\alpha_2(V_2)+\delta_5(e^2).$$
\end{prop}
\end{lem}
\begin{proof} We first verify Property \ref{property1}. Consider the $\bar{z}$-subsystem \eqref{aumet-eq1}, as $F(w)$ is Hurwitz, we define the following Lyapunov function
$$V_0(\bar{z})=\bar{z}^{\!\top}P(w)\bar{z},$$
where the positive definite matrix $P(w)$ satisfies $$P(w)F(w)+F(w)^{\!\top}P(w)=-I.$$ Let $\lambda_w$ and $\lambda_W$ be the minimum and maximum eigenvalues of $P(w)$. Taking the time derivative of $V_0(\bar{z})$ along the trajectory of the $\bar{z}$-subsystem \eqref{aumet-eq1} gives
\begin{align*}
\dot{V}_0\big|_{\eqref{aumet-eq1}}
&\leq-\alpha_0V_0+\delta_0(e^2)
\end{align*}
where $\alpha_0=\frac{3}{4\lambda_W}$ and $\delta_0(e^2)=4\|P(w)\bar{G}(e,w)e\|^2$. Thus, equation \eqref{V0} has been shown.

Equation \eqref{V1} can be done similarly as the above. Consider $\bar{\xi}$-subsystems \eqref{aumet-eq2} and \eqref{aumet-eq3}. Since $M$ is Hurwitz, define the following Lyapunov function
$$V_1(\bar{\xi})=h_0\bar{\eta}^{\!\top}{P}_M\bar{\eta}+\bar{\pi}^2$$
where $h_0$ is some positive parameter, to be determined, and the positive definite matrix ${P}_M$ satisfies $${P}_M{M}+{M}^{\!\top}{P}_M=-I.$$ Let $\lambda_p$ and $\lambda_P$ be the minimum and maximum eigenvalues of $\diag(h_0P_M, 1)$. Taking the time derivative of $V_1(\bar{\xi})$ along the trajectory of $\bar{\xi}$-subsystems \eqref{aumet-eq2} and \eqref{aumet-eq3} gives
\begin{align*}
\dot{V}_1\big|_{\eqref{aumet-eq2}+\eqref{aumet-eq3}} 
   \leq&-\alpha_1 V_1+\beta_1V_0+\delta_1(e^2)
   \end{align*}
where 
\begin{align*}
\alpha_1=&\ \frac{\min\{0.5 h_0,1\}}{\lambda_P},\ \beta_1=\frac{4\|H(w)\|^2}{\lambda_p},\ h_0=\frac{1}{2\|{P}_M N\|^2},\\
\delta_1(e^2)=&\ 4h_0\|{P}_M N\|^2\| {\myr b\left(\mu\right)^{-1}}e\|^2+4\| {\myr b\left(\mu\right)^{-1}}e+\bar{K}\left(e,\mu\right)\|^2.  
\end{align*}
Next, we consider the $\bar{\vartheta}$-subsystem \eqref{aumet-eq4}. 
Assumption \ref{ass3} guarantees that the steady-state
input signal $\mathbf{u}(\mu)$ is sufficiently rich of order $n$ ($n\in\{2s^*, 2s^*-1\}$). 
From Theorem 4.1 in \cite{liu2009parameter}, we have that $\bf{\xi}\left(\mu\right)$ in \eqref{xi-exosystem} is persistently exciting (see \cite{krstic1996invariant}). 
It is noted that $$\theta\left(\mu\right)= T(a){\bf \xi}\left(\mu\right)$$ where the matrix $T(a)$ defined in \eqref{Tmatrix} is nonsingular.
It follows from Lemma 3 in \cite{narendra1987persistent} that the vector $\theta\left(\mu\right)$ is persistently exciting. From Theorem 1 in \cite{anderson1977exponential}, the $\bar{\vartheta}$-subsystem \eqref{aumet-eq4} is exponentially stable, provided that $\bar{\gamma}(\bar{\eta}, e, \mu)=0$. Moreover, by  Theorem 4.14 in \cite{khalil2002nonlinear}, there exist a Lyapunov function $W(t,\bar{\vartheta})$ satisfying
\begin{align*}
\underline{c}_1\|\bar{\vartheta}\|^2\leq\ W(t,\bar{\vartheta})\leq&\,\, \bar{c}_1\|\bar{\vartheta}\|^2,\\
\frac{\partial W}{\partial \bar{\vartheta}}\leq&\,\, \bar{c}_2\|\bar{\vartheta}\|,\\
\frac{\partial W}{\partial t}- \frac{\partial W}{\partial \bar{\vartheta}}\left[\theta\left(\mu\right){\myr \theta\left(\mu\right)^{\!\top}}\bar{\vartheta}\right]\leq& -\bar{c}_3\|\bar{\vartheta}\|^2,
\end{align*}
where $\underline{c}_1$, $\bar{c}_1$, $\bar{c}_2$, and $\bar{c}_3$ are some positive constants. Define $V_2(\bar{\vartheta},t)=\ln\left(W(t,\bar{\vartheta})+1\right)$, which satisfies $$V_2(\bar{\vartheta},t)\leq W(t,\bar{\vartheta}). $$ Then, we have the 
$\underline{\alpha}_2(s)= \ln\left(\underline{c}_1 s^2+1\right)$ and $\bar{\alpha}_2(s)=\bar{c}_1s^2$ for all $s>0$.
Taking the time derivative of $V_2(\bar{\vartheta},t)$ along the trajectory of $\bar{\vartheta}$-subsystems \eqref{aumet-eq4} gives
\begin{align*}
\dot{V}_2\big|_{\eqref{aumet-eq4}} 
&\leq -\frac{\bar{c}_3\|\bar{\vartheta}\|^2}{W(t,\bar{\vartheta})+1}+\frac{\bar{c}_2\|\bar{\vartheta}\|\|\bar{\gamma}(\bar{\eta}, \bar{\pi}, e, \mu)\|}{W(t,\bar{\vartheta})+1}.
\end{align*}
Under Assumption \ref{ass1}, $\theta\left(\mu\right)$, $\varrho$, $ {\myr b\left(\mu\right)^{-1}}$ and $\varpi \left(\mu\right)$ are bounded. Then, it can be verified that for some positive constant $\bar{c}_4$,
\begin{align*}
\|\bar{\gamma}(\bar{\eta},\bar{\pi}, e,\mu)\|\leq&
\bar{c}_4(\|\bar{\xi}\|^2+\|\bar{\xi}\|)(1+\|\bar{\vartheta}\|)+\bar{c}_4(\|e\|+ \|e\|^2).\nonumber
 \end{align*}
Then, we have
\begin{align*}
\dot{V}_2\big|_{\eqref{aumet-eq4}} 
\leq & -\frac{\bar{c}_3\|\bar{\vartheta}\|^2}{2(\bar{c}_1\|\bar{\vartheta}\|^2+1)}+\frac{\bar{c}_2^2\bar{c}_4^2(\|e\|+ \|e\|^2)^2}{\bar{c}_3(\underline{c}_1\|\bar{\vartheta}\|^2+1)}\nonumber\\
&+\frac{\bar{c}_2^2\bar{c}_4^2(\|\bar{\xi}\|^2+\|\bar{\xi}\|)^2(1+\|\bar{\vartheta}\|)^2}{\bar{c}_3(\underline{c}_1\|\bar{\vartheta}\|^2+1)}.
\end{align*}
For any constant $s\geq 0$, $\underline{c}_1s^2+1\geq 1$ and $ \frac{\bar{c}_2^2\bar{c}_4^2(1+s)^2}{\bar{c}_3(\underline{c}_1s^2+1)}$ is bounded for some positive constant $\bar{c}_5$. Then, from the inequality $(a+b)^2\leq 2(a^2+b^2)$, we get
\begin{align}\label{V2alpha}
\dot{V}_2\big|_{\eqref{aumet-eq4}} 
\leq & -\frac{\bar{c}_3\|\bar{\vartheta}\|^2}{2(\bar{c}_1\|\bar{\vartheta}\|^2+1)}+2\frac{\bar{c}_2^2\bar{c}_4^2}{\bar{c}_3}(\|e\|^2+ \|e\|^4)\nonumber\\
&+2\bar{c}_5(\|\bar{\xi}\|^4+\|\bar{\xi}\|^2)\nonumber\\
\leq & -\alpha_2(V_2)+\delta_3(V_1)V_1+\delta_2(e^2)
\end{align}
where 
\begin{align*}
\alpha_2(s)=&\ \frac{\bar{c}_3 \underline{c}_1(e^s-1)}{2(\bar{c}_1^2(e^s-1)+\bar{c}_1\underline{c}_1)}, &\delta_3(s)=2\bar{c}_5(\underline{\alpha}_1^{-2}s^2+\underline{\alpha}_1^{-1}s),\\
\delta_2(e^2)=&\ 2\frac{\bar{c}_2^2\bar{c}_4^2}{\bar{c}_3}(\|e\|^2+ \|e\|^4). &
\end{align*}
We now verify Property \ref{property2}. Consider the function $\bar{g}(\bar{z}, e, \bar{\eta}, \bar{\vartheta},\mu)$ in \eqref{aumet}, using the inequality $(a+b)^2\leq 2a^2+2b^2$, we obtain
\begin{align*}
\|\bar{g}(\bar{z}, e, \bar{\eta}, \bar{\vartheta},\mu)\|^2 \leq &\ 2\|H(w)\|^2\|\bar{z}\|^2+ 2\|\bar{K}\left(e,\mu\right)\|^2\\
&+2{\myr b(\mu)^2}\|\chi_s({\eta}, {\vartheta})-\chi_s\left(\theta\left(\mu\right), \varrho\right)\|^2.
\end{align*}
It can be verified that the function $\bar{g}(\bar{z}, e, \bar{\eta}, \bar{\vartheta},\mu)$ is smooth and vanishes at $\col(\bar{z}, e, \bar{\eta}, \bar{\vartheta})=\col(0,0,0,0)$. Let $$\zeta=\col(\eta, \vartheta),\quad \zeta_0=\col(\theta\left(\mu\right), \varrho)\quad \textnormal{and}\quad \bar{\zeta}=\zeta-\zeta_0.$$ From equation \eqref{chisatu}, the function $\chi_s(\zeta)$ is bounded for all $\zeta\in \mathds{R}^{2n}$, by the \cite[Lemma A.1]{xu2016output} and \cite[Remark A.1]{xu2016output}, there exist $\gamma_1, \gamma_2\in {\myr \mathcal{K}_o}\cap \mathcal{O}(\emph{Id})$ such that
\begin{align*}
2{\myr b(\mu)^2}|\chi_s(\zeta_0+\bar{\zeta})-\chi_s\left(\zeta_0\right)|^2\leq \gamma_1(\|\bar{\eta}\|^2)+\gamma_2(\|\bar{\vartheta}\|^2),
\end{align*}
for $\forall\zeta\in \mathds{R}^{2n}$ and $\mu\in \mathds{V}\times \mathds{W}\times \mathds{S}$.
From equation \eqref{V2alpha}, it also can be verified that
\begin{align*}\limsup_{s\rightarrow 0^{+}}\frac{\gamma_2\circ [(e^s-1)\underline{c}_1^{-1}]}{\alpha_2(s)}=&\limsup_{s\rightarrow 0^{+}}\frac{\gamma_2\circ [(e^s-1)\underline{c}_1^{-1}]}{(e^s-1)\underline{c}_1^{-1}}\\
&\times\limsup_{s\rightarrow 0^{+}}\frac{(e^s-1)\underline{c}_1^{-1}}{\alpha_2(s)}< +\infty.\end{align*}
Hence, there exists a positive constant $\phi_2$ such that
\begin{align*}
2\|b(\mu)\|^2|\chi_s(\zeta_0+\bar{\zeta})-\chi_s\left(\zeta_0\right)|^2\leq \gamma_1(\|\bar{\eta}\|^2)+\phi_2\alpha_2(V_2),
\end{align*}
for $\forall\zeta\in \mathds{R}^{2n}$ and $\mu\in \mathds{V}\times \mathds{W}\times \mathds{S}$.
From equation \eqref{V1}, we have
\begin{align*}\limsup_{s\rightarrow 0^{+}}\frac{\gamma_1\circ [s\underline{\alpha}_1^{-1}]}{s}=&\limsup_{s\rightarrow 0^{+}}\frac{\gamma_1\circ[ s\underline{\alpha}_1^{-1}]}{s\underline{\alpha}_1^{-1}}\underline{\alpha}_1< +\infty.\end{align*}
Hence, from \cite[Lemma A.3]{xu2016output}, there exists a positive constant $\phi_1$ such that
\begin{align*}
2{\myr b(\mu)^2}|\chi_s(\zeta_0+\bar{\zeta})-\chi_s\left(\zeta_0\right)|^2\leq \phi_1 V_1+\phi_2\alpha_2(V_2),
\end{align*}
for $\forall\zeta\in \mathds{R}^{2n}$ and $\mu\in \mathds{V}\times \mathds{W}\times \mathds{S}$.
The function $2\|\bar{K}\left(e,\mu\right)\|^2$  is a continuously differentiable function satisfying $2\|\bar{K}\left(0,\mu\right)\|^2=0$ for any $\mu\in \mathds{V}\times \mathds{W}\times \mathds{S}$. By (iv) of \cite[Lemma 11.1]{chen2015stabilization}, there exists a smooth function $\delta_6(e^2)$ such that
$$2\|\bar{K}\left(e,\mu\right)\|^2\leq \delta_6(e^2)e^2\equiv \delta_5(e^2).$$
Then, from equation \eqref{V0}, we will have
\begin{align*}
\|\bar{g}(\bar{z}, e, \bar{\eta}, \bar{\vartheta},\mu)\|^2 \leq &\ \phi_0 V_0+\phi_1V_1+\phi_2\alpha_2(V_2)+\delta_5(e^2),
\end{align*}
where $\phi_0=2\|H(w)\|^2\underline{\alpha}_0^{-1}$.
\end{proof}
\begin{rem}
$\alpha_2(\cdot)\in \mathcal{K}_{o}$ in Property \ref{property1} of Lemma \ref{lemmabodev} means that the $\bar{\vartheta}$-subsystem of \eqref{aumet} is iISS but not Input-to-State Stable (ISS) (which would require a stronger gain condition for $\alpha_{2}(\cdot)\in\mathcal{K}_{\infty}$ in \eqref{V2}). Property \ref{property1} also establishes
the growth conditions for the nonlinearity in \eqref{aumet}.
\end{rem}

\subsection{Lie Bracket Approximations}
Before we present our main results, we first review some content related to the Lie bracket averaging approach. Let us introduce a control system in the following nonlinear affine form:
\begin{align}\label{general}
\dot{x} = f(x) +\sum\nolimits_{i=1}^{m} g_{i}(x)\sqrt{\omega}u_i(\omega t)
\end{align}
where $x\in \mathds{R}^n$, $x(0)\in \mathds{R}^{n}$, $\omega>0$, $t\in [0,\infty)$, $f(x)$ and $g_i(x)$ are twice continuously differentiable. For $i=1,\dots,m$, the input function $u_i(\omega t)$ are assumed to be uniformly bounded and periodic with period $T$ such that $\int_{0}^{T}u_i(\omega \tau)d\tau=0$.
\begin{rem} Please note that the dynamical system \eqref{general} provides a generic representation of the class of systems to which belongs the closed-loop system presented in the next section with $x\equiv X=\textnormal{\col}(\bar{z}, \bar{\eta}, \bar{\pi}, \bar{\vartheta},  e)$ and $f(x)\equiv P(X,\mu)$ (defined at the beginning of the next section).\end{rem}

Following the works \citep*{gurvits1992averaging,durr2013lie}, the Lie bracket average of nonlinear system \eqref{general} can be calculated in the following form:
\begin{align}\label{avgeneral}
\dot{\tilde{x}}=&\ f(\tilde{x})\nonumber\\
&+\frac{1}{T}\sum\limits_{i< j}[g_i,g_j](\tilde{x})\int_{0}^{T}\int_{0}^{\theta}u_j(\omega \theta)u_i(\omega \tau)d\tau d\theta.
\end{align}
We now define the nonlinear parameterized dynamical system:
\begin{align}\label{nonpara}
\dot{x}^{\epsilon}=F^{\epsilon}(t,x^{\epsilon})
\end{align}
with a small positive parameter $\epsilon$. The solution of \eqref{nonpara} is denoted by $x^{\epsilon}(t)=\phi_{\epsilon}(t,t_0,x_0)$, where $\phi_{\epsilon}$ is the flow of the system for $t>0$ with initial conditions $t_0$, $x^{\epsilon}(t_0)=x^{\epsilon}_0$. The averaged dynamics are defined as follows:
\begin{align}\label{avnonpara}
\dot{x} =F (t,x)
\end{align}
whose solution of \eqref{avnonpara} is denoted by $x (t)=\phi (t,t_0,x_0)$, where $\phi$ is the flow of the system for $t>0$ with initial conditions $t_0$, $x(t_0)=x_0$. The  convergence property is defined as follows:
\begin{defi}\citep*{moreau2000practical} The systems \eqref{nonpara} and \eqref{avnonpara} are said to satisfy the convergence property if for every $T\in (0,\infty)$ and compact set $\mathds{K}\in \mathds{R}^{n}$ satisfying
$\left\{(t,t_0,x_0)\in \mathds{R}\times \mathds{R}\times \mathds{R}^n: t\in [t_0,t_0+T],x_0\in \mathds{K}\right\}\subset \textnormal{Dom}\;\phi$, for every $\delta\in (0, \infty)$ there exists $\epsilon^*$ such that for all $t_0\in \mathds{R}$, for all $x_0\in \mathds{K}$ and for all $\epsilon\in (0, \epsilon^*)$, $$\|\phi^{\epsilon}(t,t_0,x_0)-\phi (t,t_0,x_0)\|<\delta,~~~~\forall t\in [t_0,t_0+T].$$
\end{defi}
Then, we recall the $\epsilon$-Semi-global practical uniform asymptotic stability ($\epsilon$-SPUAS).
\begin{defi}[$\epsilon$-SPUAS] An equilibrium point of \eqref{nonpara} is said to be $\epsilon$-SPUAS if it satisfies uniform stability, uniform boundedness and global uniform attractivity.
\end{defi}
Then, systems \eqref{general} and \eqref{avgeneral} satisfy the following lemma.
\begin{lem}\label{morelemm}\citep*{moreau2000practical} Assume that systems \eqref{nonpara} and \eqref{avnonpara} satisfy the converging trajectories propriety. If the origin of system \eqref{avnonpara} is a global uniform asymptotically stable equilibrium point, then the origin of system \eqref{nonpara} is $\epsilon$-SPUAS.
\end{lem}

\begin{figure*}[htp]
    \centering
\usetikzlibrary{calc,shadows}
\tikzset{ashadow/.style={opacity=.25, shadow xshift=0.07, shadow yshift=-0.07},}
\tikzstyle{block} = [draw, rectangle, minimum height=2em, minimum width=6em, drop shadow={ashadow, color=brown!60!gray}]
\tikzstyle{sum} = [draw, fill=blue!20, circle, node distance=1cm, drop shadow={ashadow, color=brown!60!gray}]
\tikzstyle{input} = [coordinate]
\tikzstyle{output} = [coordinate]
\tikzstyle{pinstyle} = [pin edge={to-,thick,black}]
\begin{tikzpicture}[auto, node distance=2cm,>=latex']
    \node [input, name=input] {};
    \node [sum, right=2cm of input] (sum) {};
    \node [block, fill=brown!30, right=1.5cm of sum] (controller) {\footnotesize$\begin{array}{c}
         \text{Extremum-Seeking controller:}  \\ u=\sqrt{\alpha\omega}\cos\left(\omega t+ke^2\right) \rho(e)+\chi_s(\eta, \vartheta) \end{array}$};
    \node [block, fill=blue!15, right=3.3cm of controller, pin={[pinstyle]below:{\small External disturbances}},
            node distance=3cm] (system) {Plant};
    \draw [->] (controller) -- node[name=u]{$u$} (system);
    \node [output, right=1.5cm of system] (output) {};
    \node [block, fill=green!5, below=1cm of u] (filter0) {\footnotesize$\begin{array}{c}
         \text{Input filter:}  \\ \dot{\pi}=-\pi+u \end{array}$};
    \node [block, fill=brown!20, left =1.3cm of filter0] (IM) {\footnotesize$\begin{array}{c}
         \text{Internal model:}  \\ \dot{\eta}= M\eta + N\pi \end{array}$};
    \node [block, fill=green!15, below =0.5cm of IM] (estimator) {\footnotesize$\begin{array}{c}
         \text{Parameter estimator:}  \\ \dot{\vartheta}= -\Theta\eta[\eta^{\!\top}\vartheta-\pi] \end{array}$};

    \draw [draw,thick,->] (input) -- node[above]{$r=q(v,w)$}node[below]{\footnotesize Reference} node[below,pos=0.99] {$-$} (sum);
    \draw [thick,->] (sum) -- node {$e$} (controller);
    \draw [thick,->] (system) -- (output);
    \draw [thick,->] ($(system.east)+(0.5cm, 0cm)$) -- ($(system.east)+(0.5cm, 1cm)$) --($(system.east)+(-6cm, 1cm)$)node [below] {$y$}-|  (sum);
    \draw [thick,->] (u) -- (filter0); 
    \draw [thick,->] (filter0) --node[above]{$\pi$} (IM); 
    \draw [thick,->] (IM) --node[right]{\footnotesize$\eta$} (estimator);
    \draw [thick,->] ($(filter0.west)+(-0.5cm, 0cm)$) |- (estimator);
    \draw [thick,->] ($(IM.west)+(0cm, 0cm)$) -|node[right,near end]{\footnotesize$\eta$} ($(controller.south)+(-1.5cm, 0cm)$);
    \draw [thick,->] ($(estimator.west)+(0cm, 0cm)$) -|node[right,near end]{\footnotesize$\vartheta$} ($(controller.south)+(-2cm, 0cm)$);
\end{tikzpicture} 
\caption{Extremum-Seeking Nonlinear Regulator with Concurrent Uncertainties in
 Exosystems and Control Directions}\label{framework}
\end{figure*}

\section{Main Results} \label{section4} 

\subsection{Extremum-Seeking Control Design}
In this section, we proposed using an extremum-seeking control approach to handle the unknown control direction  (see Figure \ref{framework}).
Let $X=\col(\bar{z}, \bar{\eta}, \bar{\pi}, \bar{\vartheta},  e)$ and 
$$P(X,\mu)=\left[
          \begin{array}{@{}c@{}}
           F\left(w\right)\bar{z}+\bar{G}(e,w)e \\
            M  \bar{\eta} + \bar{\varepsilon}(\bar{\pi}, e)\\
             -\bar{\pi}-\bar{\delta}(\bar{z}, e,\mu)\\
             -\Theta\theta\left(\mu\right){\myr \theta\left(\mu\right)^{\!\top}}\bar{\vartheta}+ \bar{\gamma}(\bar{\eta}, \bar{\pi}, e,  \mu)\\
             \bar{g}(\bar{z}, e, \bar{\eta}, \bar{\vartheta},\mu)
          \end{array}
        \right].$$

\begin{thm}\label{thm1} Under Assumptions \ref{ass1}--\ref{ass4}, there exist a smooth positive function $\rho\left(\cdot\right)^2\geq 1$ and some sufficiently large positive constant ${k}$ and $\alpha$, and a dynamic output feedback controller
\begin{subequations}\label{escun}\begin{align}
{u}=&\ \sqrt{\alpha\omega}\cos\left(\omega t+ke^2\right) \rho(e)+\chi_s(\eta, \vartheta),\\
\dot{\eta}=&\ M\eta+N\pi,\\
  \dot{\pi}=&-\pi+u,\\
  \dot{\vartheta}=&-\Theta\eta[\eta^{\!\top}\vartheta-\pi],
\end{align}\end{subequations}
solves the robust PORP for the closed-loop system composed of \eqref{chisatu}, \eqref{aumet} and \eqref{escun}.
\end{thm}
\begin{proof} 
The error dynamics \eqref{aumet} with the extremum-seeking control \eqref{escun} can be expanded as
\begin{align*}
\dot{X}=&\ P(X,\mu)+\left(\underbrace{\left[
                  \begin{array}{@{}c@{}}
                 \textbf{0}_{4\times 1}\\
b(\mu)\sqrt{\alpha\omega}\cos\left(ke^2\right)\rho(e)\\
                  \end{array}
                \right]}_{a_1(X)}\cos(\omega t)\right.\\
&\left.\underbrace{-\left[
                  \begin{array}{@{}c@{}}
                   \textbf{0}_{4\times 1}\\
b(\mu)\sqrt{\alpha\omega}\sin\left(ke^2\right)\rho(e)\\
                  \end{array}
                \right]}_{a_2(X)}\sin(\omega t)\right).
\end{align*}
Then, in line with the Lie bracket average formula from system \eqref{general} to \eqref{avgeneral}, we have the corresponding Lie-bracket averaged system as follows:
\begin{align*}
\dot{\tilde{X}}=&\  P(\tilde{X},\mu)\\
&+\frac{1}{T}[a_1(\tilde{X}), a_2(\tilde{X})]\int_{0}^{T}\int_{0}^{\theta}\cos(\omega \theta)\sin(\omega \tau)d\tau d\theta.
\end{align*}
where
\begin{align*}
[a_1(\tilde{X}), a_2(\tilde{X})]
&=2\omega\left[
                  \begin{array}{@{}c@{}}
                  \textbf{0}_{4 \times 1}\\
-k{\myr b(\mu)^2}{\myr\rho\left(\tilde{e}\right)^2}\tilde{e}{\alpha}\\
                  \end{array}
                \right],\\
\frac{1}{T}\int_{0}^{T}\int_{0}^{\theta}\cos(\omega \theta)\sin(\omega \tau)d\tau d\theta&={-\frac{1}{2\omega}\myr }.
\end{align*}
Then, we have the following averaged system
\begin{subequations}\label{aveescun1}\begin{align}
  \dot{\tilde{z}} &= F\left(w\right)\tilde{z}+\bar{G}(\tilde{e},w)\tilde{e}, \\
    \dot{\tilde{\eta}}&= M  \tilde{\eta} + \bar{\varepsilon}(\tilde{\pi}, e),\\
       \dot{\tilde{\pi}}& =-\tilde{\pi}-\bar{\delta}(\tilde{z},\tilde{ e},\mu),\\
    \dot{\tilde{\vartheta}}& =-\Theta\theta\left(\mu\right){\myr \theta\left(\mu\right)^{\!\top}}\tilde{\vartheta}+ \bar{\gamma}(\tilde{\eta}, \tilde{\pi}, \tilde{e}, \mu),\\
  \dot{\tilde{e}}&=\bar{g}(\tilde{z}, \tilde{e}, \tilde{\eta}, \tilde{\vartheta},\mu)-k{\myr b(\mu)^2}{\myr\rho\left(\tilde{e}\right)^2}\tilde{e}{\alpha}.
\end{align}\end{subequations}
Following Lemma \ref{lemmabodev}, we pose the Lyapunov function $ U_1:=U_1(\tilde{z}, \tilde{\xi})$ defined by
\begin{align*}
U_1(\tilde{z}, \tilde{\xi})=\epsilon_0V_0(\tilde{z})+V_1(\tilde{\xi}),
\end{align*}
where $\tilde{\xi}=\col(\tilde{\eta}, \tilde{\pi})$, $\epsilon_0 $ is any positive constant bigger than $\alpha_1+\beta_1/\alpha_0$ with $\alpha_0$, $\alpha_1$ and $\beta_1$ obtained from Property \ref{property1}. Then, it can be verified that $$\min\{\epsilon_0, 1\}( V_0 + V_1) \leq U_1\leq \max\{\epsilon_0, 1\}( V_0 + V_1).$$
From Property \ref{property1} in Lemma \ref{lemmabodev}, the time derivative of $U_1(t)$ along \eqref{aveescun1} can be evaluated as
\begin{align}\label{U1V0V1}
\dot{U}_1\leq& -\alpha_1( V_0 + V_1)+\delta_1(\tilde{e}^2)+\epsilon_0\delta_0(\tilde{e}^2)\nonumber\\
\leq& -\alpha_u U_1+\delta_u(\tilde{e}^2),
\end{align}
where $\alpha_u=\frac{\alpha_1}{\max\{\epsilon_0, 1\}}$ and $\delta_u(e^2)=\delta_1(\tilde{e}^2)+ \epsilon_0\delta_0(\tilde{e}^2)$.
We define a Lyapunov-like function $ U:= U (t,\tilde{z}, \tilde{\xi},\tilde{\vartheta}, \tilde{e})$ of the form
 \begin{align*}
 U(t)=\epsilon_1 U_1+\int_{0}^{U_1}\kappa(s)ds+\epsilon_2 V_2(\tilde{\vartheta})+\frac{1}{2}\tilde{e}^2
\end{align*}
 where the positive function $\kappa(\cdot)$ is specified later. From Property \ref{property1} in Lemma \ref{lemmabodev} and the equation \eqref{U1V0V1}, the time derivative of $U(t)$ along \eqref{aveescun1} can be evaluated as
\begin{align}
\dot{U}\leq& -\epsilon_1\alpha_u U_1+\epsilon_1\delta_u(\tilde{e}^2)+\kappa \circ U_1 \left[-\alpha_u U_1+\delta_u(\tilde{e}^2)\right]\nonumber\\
&-\epsilon_2\alpha_2(V_2) +\epsilon_2\delta_3(V_1)+\epsilon_2\delta_2(\tilde{e}^2)\nonumber\\
&+ 0.25\tilde{e}^2+\|\bar{g}(\tilde{z}, \tilde{e}, \tilde{\eta}, \tilde{\vartheta},\mu)\|^2-k{\myr b(\mu)^2}{\myr\rho\left(\tilde{e}\right)^2}\tilde{e}^2{\alpha}.\nonumber
\end{align}
From Property \ref{property2} in Lemma \ref{lemmabodev}, we have
\begin{align*}
\dot{U}\leq& -\epsilon_1\alpha_u U_1+\frac{\max\{\phi_0, \phi_1\}}{\min\{\epsilon_0, 1\}}U_1\nonumber\\
&+\kappa \circ U_1 \left[-\alpha_u U_1+\delta_u(\tilde{e}^2)\right] +\epsilon_2{\myr \delta_3(V_1)}\nonumber\\
&-[\epsilon_2-\phi_2]\alpha_2(V_2)+\epsilon_2\delta_2(\tilde{e}^2)+\delta_5(\tilde{e}^2)\nonumber\\
&+\epsilon_1\delta_u(\tilde{e}^2)+ 0.25\tilde{e}^2-k{\myr b(\mu)^2}{\myr\rho\left(\tilde{e}\right)^2}\tilde{e}^2{\alpha}.
\end{align*}
Since $\delta_3(0)=0$, by Lemma 7.8 in \cite{huang2004nonlinear} and from equation \eqref{V2alpha}, ${\myr\delta_3(s)}\leq \bar{\delta}_3(s) s$ where $ \bar{\delta}_3(s)=2\bar{c}_5(\underline{\alpha}_1^{-2}s+\underline{\alpha}_1^{-1})$ is a known smooth positive function. By using the change of supply rate technique in \cite{sontag1995changing} and Lemma 2.1 in \cite{xu2017global}, we can choose any smooth function $\kappa(s)$ such that $$\bar{\kappa}(s)>1+\epsilon_2\bar{\delta}_3(s)+\frac{\max\{\phi_0, \phi_1\}}{\min\{\epsilon_0, 1\}},$$ where $\bar{\kappa}(s)=\alpha_u\epsilon_1+\frac{\alpha_u}{2} \kappa\circ(\alpha_u s)$. {\myr Moreover, the definition of $U_1$ further implies  $U_1\geq V_1$}. Then we have
\begin{align}
\dot{U}\leq& -\bar{\kappa}(U_1)U_1+\bar{\delta}_u(\tilde{e}^2)\nonumber\\
& +\epsilon_2\bar{\delta}_3(U_1)U_1+\frac{\max\{\phi_0, \phi_1\}}{\min\{\epsilon_0, 1\}}U_1\nonumber\\
&-[\epsilon_2-\phi_2]\alpha_2(V_2)+\epsilon_2\delta_2(\tilde{e}^2)+\delta_5(e^2)\nonumber\\
&+ 0.25\tilde{e}^2-k{\myr b(\mu)^2}{\myr\rho\left(\tilde{e}\right)^2}\tilde{e}^2{\alpha}\nonumber\\
\leq&-U_1-[\epsilon_2-\phi_2]\alpha_2(V_2)+\bar{\delta}_u(\tilde{e}^2)\nonumber\\
&+\epsilon_2\delta_2(\tilde{e}^2)+\delta_5(\tilde{e}^2)+ 0.25\tilde{e}^2-k{\myr b(\mu)^2}{\myr\rho\left(\tilde{e}\right)^2}\tilde{e}^2{\alpha}.\nonumber
\end{align}
where $\bar{\delta}_u(\tilde{e}^2)=\left[\epsilon_1+\kappa\circ\left(2\delta_u(\tilde{e}^2)\right)\right]\delta_u(\tilde{e}^2)$.
Since $$\bar{\delta}_u(0)+\delta_2(0)+\delta_5(0)=0,$$ by using Lemma 11.1 in \cite{chen2015stabilization}, and from Lemma \ref{lemmabodev}, we have $$\bar{\delta}_u(s)+\delta_2(s)+\delta_5(s)\leq \Delta_M {\Delta}(s) s$$ for any $s\geq 0$ and known positive smooth function ${\Delta}(\cdot)$ and positive constant $\Delta_M$. Then, we have
\begin{align*}
\dot{U}\leq&-U_1-[\epsilon_2-\phi_2]\alpha_2(V_2)+{\myr \Delta_M\myr} {\Delta}\left(\tilde{e}^2\right) \tilde{e}^2\nonumber\\
&+ 0.25\tilde{e}^2-k{\myr b(\mu)^2}{\myr\rho\left(\tilde{e}\right)^2}\tilde{e}^2{\alpha}.
\end{align*}
Letting the smooth function $\rho(\cdot)$, and the positive numbers $\epsilon_2$ and ${k}$ be such that  ${\myr \rho(\tilde{e})^2\geq \max \{1, {\Delta}\left(\tilde{e}^2\right)\} }$, $\epsilon_2\geq\phi_2+1$ and $k\alpha\geq \frac{\myr 1.25+\Delta_M}{{\myr b(\mu)^2}}$, gives
\begin{align}\label{strlyafun}\dot{U}\leq -U_1-\alpha_2(V_2)-\tilde{e}^2.\end{align}
Since $U(t, \tilde{z}, \tilde{\xi},\bar{\vartheta}, \tilde{e})$ is positive definite, radially unbounded and satisfies inequality \eqref{strlyafun},
it follows that system \eqref{aveescun1} is globally uniformly asymptotically stable.
By using Corollary 1 in \cite{scheinker2012minimum} and Lemma \ref{morelemm}, we have that the error dynamics \eqref{aumet} with the extremum-seeking control \eqref{escun} is $\frac{1}{\omega}$-semi-globally uniformly asymptotically stable, which further implies that
there exists a constant $\nu(\frac{1}{\omega})$ and a $\omega^*$ such that for all initial conditions in some compact set and $ v(0)\in \mathds{V}$ and $\omega> \omega^*$, the nominal trajectories are such
that $$\|\col( \bar{z}, \bar{\xi},\bar{\vartheta}, \bar{e})-\col( \tilde{z}, \tilde{\xi},\tilde{\vartheta}, \tilde{e})\|<\nu(1/\omega).$$ This completes the proof.
\end{proof}

\begin{rem}\myr  The original error system \eqref{aumet} is subject to unknown uncertainties, denoted by $\mu$, and an unknown control direction $b(\mu)$ with unknown time-varying coefficients.
By employing an extremum-seeking control (ESC) approach and Lie bracket averaging technique, system \eqref{aumet} is averaged to system \eqref{aveescun1} with a positive control direction $b(\mu)^2$, omitting high-order terms. \end{rem}
\begin{rem}\myr
The nonlinear component in controller \eqref{escun} can potentially generate inputs that exceed the actuator's input range, leading to high-gain feedback. 
To address this issue, motivated by the extremum-seeking control approach proposed in \cite{scheinker2012minimum,dehaan2005extremum}, we introduce the following controller \eqref{escun2} that leverages the properties of trigonometric functions. 
This proposed controller eliminates the requirement for dynamic gain, mitigates the risk of high-gain effects, and ensures bounded control actions.
 \end{rem}
\begin{thm} Under Assumptions \ref{ass1}--\ref{ass4}, there exist smooth a positive function $\rho \left(\cdot\right)\geq 1$ and some sufficiently large positive constants ${k}$ and $\alpha$, and a dynamic output feedback controller
\begin{subequations}\label{escun2}\begin{align}
{u}=&\ \sqrt{\alpha\omega}\cos\left(\omega t+k\int_{0}^{e^2} \rho(s)ds\right)+\chi_s(\eta, \vartheta),\\
\dot{\eta}=&\ M\eta+N\pi,\\
  \dot{\pi}=&-\pi+u,\\
  \dot{\vartheta}=&-\Theta\eta[\eta^{\!\top}\vartheta-\pi],
\end{align}\end{subequations}
solves the robust PORP for the closed-loop system composed of \eqref{chisatu}, \eqref{aumet} and \eqref{escun2}.
\end{thm}
\begin{proof} 

The error dynamics \eqref{aumet} with the extremum-seeking control \eqref{escun2} can be expanded as
\begin{align*}
\dot{X}=&\  P(X,\mu)+\left(\underbrace{\left[
                  \begin{array}{@{}c@{}}
                                       \textbf{0}_{4\times 1} \\
b(\mu)\sqrt{\alpha\omega}\cos\left(k\int_{0}^{e^2} \rho(s)ds\right)\\
                  \end{array}
                \right]}_{b_1(X)}\cos(\omega t)\right.\\
&\left.\underbrace{-\left[
                  \begin{array}{@{}c@{}}
                   \textbf{0}_{4\times 1} \\
b(\mu)\sqrt{\alpha\omega}\sin\left(k\int_{0}^{e^2} \rho(s)ds\right) \\
                  \end{array}
                \right]}_{b_2(X)}\sin(\omega t)\right).
\end{align*}
Then, the corresponding Lie-bracket averaged system can be calculated as
\begin{align*}
\dot{\tilde{X}}=&\ P(\tilde{X},\mu)+\frac{1}{T}[b_1(\tilde{X}), b_2(\tilde{X})]\int_{0}^{T}\int_{0}^{\theta}\cos(\omega \theta)\sin(\omega \tau)d\tau d\theta.
\end{align*}
where
\begin{align*}
[b_1(\tilde{X}), b_2(\tilde{X})]
&=2\omega\left[
                  \begin{array}{@{}c@{}}
                  \textbf{0}_{4\times 1} \\
-k{\myr b(\mu)^2}\rho\left(\tilde{e}^2\right)\tilde{e}{\alpha} \\
                  \end{array}
                \right],\\
\frac{1}{T}\int_{0}^{T}\int_{0}^{\theta}\cos(\omega \theta)\sin(\omega \tau)d\tau d\theta&=-\frac{1}{2\omega}.
\end{align*}
Then, we have the following averaged system
\begin{align*}
  \dot{\tilde{z}} &= F\left(w\right)\tilde{z}+\bar{G}(\tilde{e},w)\tilde{e}, \\
    \dot{\tilde{\eta}}&= M  \tilde{\eta} + \bar{\varepsilon}(\tilde{\pi}, e),\\
       \dot{\tilde{\pi}}& =-\tilde{\pi}-\bar{\delta}(\tilde{z},\tilde{ e},\mu),\\
    \dot{\tilde{\vartheta}} &=-\Theta\theta\left(\mu\right){\myr \theta\left(\mu\right)^{\!\top}}\tilde{\vartheta}+ \bar{\gamma}(\tilde{\eta}, \tilde{\pi}, \tilde{e}, \mu),\\
  \dot{\tilde{e}}&=-k\alpha {\myr b(\mu)^2}\rho\left(\tilde{e}^2\right)\tilde{e}+\bar{g}(\tilde{z}, \tilde{e}, \tilde{\eta}, \tilde{\vartheta},\mu).
\end{align*}
 The rest of the proof proceeds following
the developments from the proof of Theorem \ref{thm1}. It is thus
omitted for the sake of brevity.
\end{proof}
\begin{rem} From the development, the steady-state input $\mathbf{u}(v,w,\sigma)$ is a function of the system dynamics and the exosystem  with concurrent uncertainties in
the exosystem and the control direction. Therefore, the computation of an explicit solution of the internal mode for the steady-state generator \eqref{xi-exosystem} would be extremely difficult or impossible. This could be addressed by considering learning techniques such as in \cite{zisis2021control,wang2023nonparametric}. By employing the non-adaptive and non-Nussbaum-type framework, we can avoid the need for an explicit solution to the internal model.
    \end{rem}
 
\section{Simulation Example}\label{section5}
   \begin{figure}[ht]
 \epsfig{figure=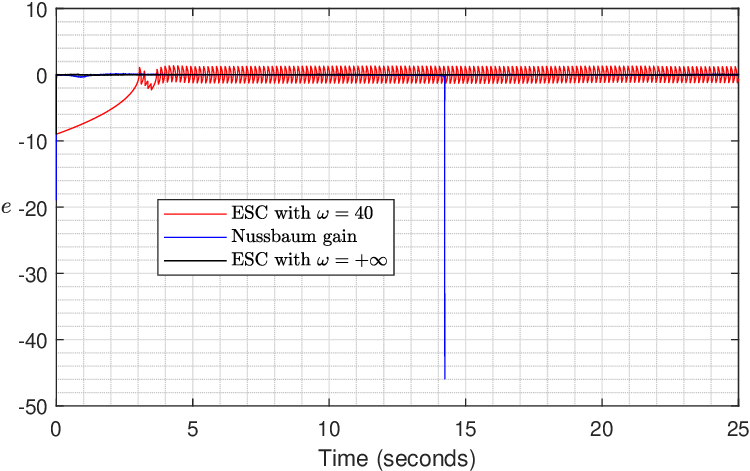,width=0.48\textwidth}
 \caption{Tracking error in terms of the Nussbaum gain technique \eqref{nussbaum} and ESC approach \eqref{escun} over different frequency}\label{fige}
\end{figure}
\begin{figure}[ht]
 \epsfig{figure=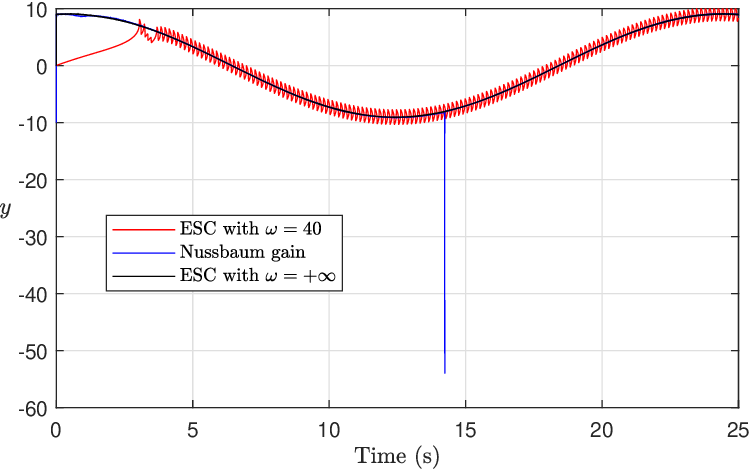,width=0.48\textwidth}
\caption{Trajectory $y$ in terms of the Nussbaum gain technique \eqref{nussbaum} and ESC approach \eqref{escun} over different frequency}\label{figy}
\end{figure}
   \begin{figure}[ht]
\epsfig{figure=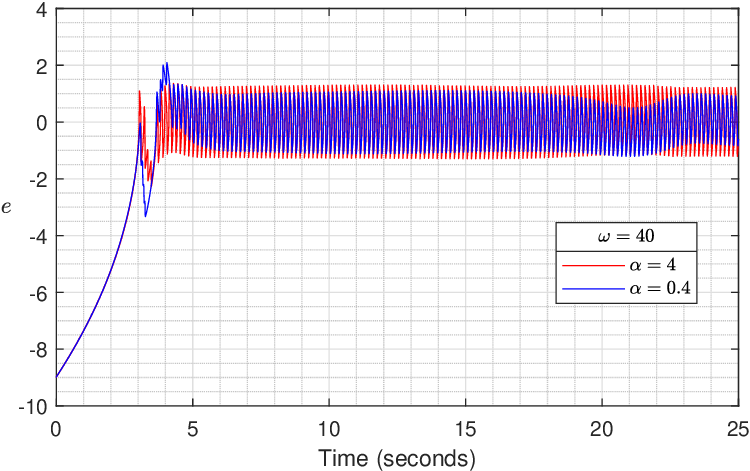,width=0.48\textwidth}
\caption{Tracking error subject to ESC approach \eqref{escun} in terms of different $\alpha$}\label{figalpha}
\end{figure}
  \begin{figure}[ht]
 \epsfig{figure=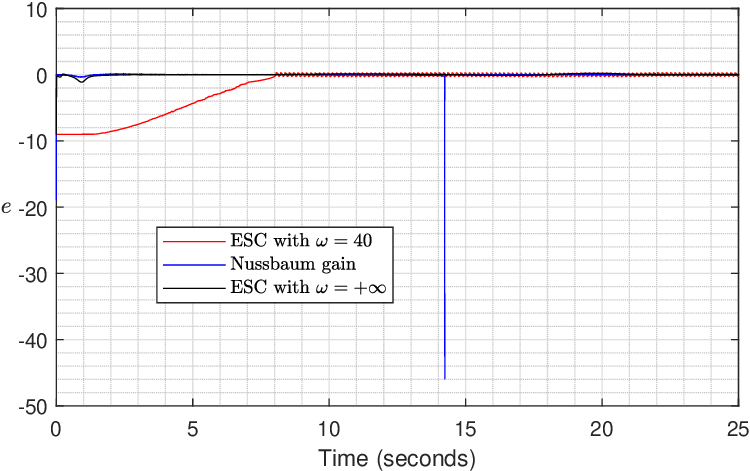,width=0.48\textwidth}
 \caption{Tracking error in terms of the Nussbaum gain technique \eqref{nussbaum} and ESC approach \eqref{escun2} over different frequency}\label{fige2}
 \end{figure}
\begin{figure}[ht]
  \epsfig{figure=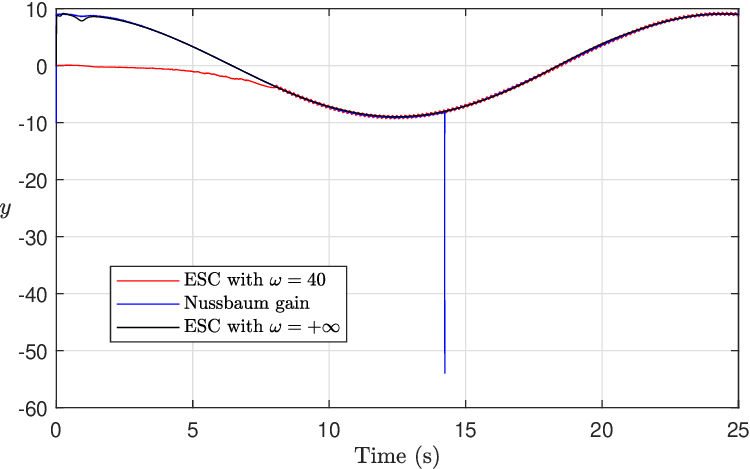,width=0.48\textwidth}
\caption{Trajectory $y$ in terms of the Nussbaum gain technique \eqref{nussbaum} and ESC approach \eqref{escun2} over different frequency}\label{figy2}
\end{figure}

   \begin{figure}[ht]
\epsfig{figure=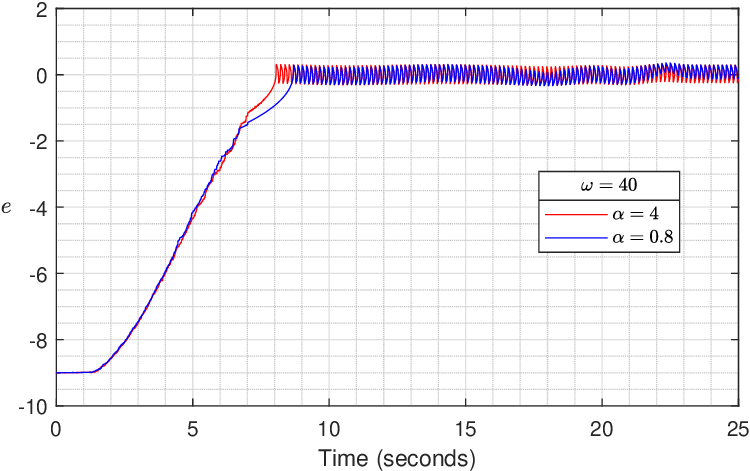,width=0.48\textwidth}
\caption{Tracking error subject to ESC approach \eqref{escun2} in terms of different $\alpha$}\label{figalpha2}
\end{figure}

In this simulation example, we consider the following nonlinear output feedback system, taken from \cite{liu2009parameter},
\begin{align}
\dot{z}&=\left[
                \begin{array}{@{}cc@{}}
                  -1 & 0\\
                  0& -1 \\
                \end{array}
              \right]z +\left[
                           \begin{array}{@{}c@{}}
                             (\sin(y-v_1))^2y \\
                            y \\
                           \end{array}
                         \right],\nonumber\\
 \dot{y}&=\left[0, 1\right]z -w_1y-w_2y^3 + b(v,w)u,  \nonumber\\
 e &=y-v_1,\label{example-put-feedback}
\end{align}
where $z =\textnormal{\col}\left(z_{1},z_{2}\right)$ and $y$ are the state variables, $w=\col(w_1, w_2$) is the unknown parameter vector, and $b(v,w)$ is the time-varying coefficient.
We assume that the uncertainty $w \in \mathds{W}\subseteq\mathds{ R}^2$. The following exosystem generates the signal $v$:
\begin{align*}
\dot{v}=\left[
       \begin{array}{@{}cc@{}}
         0 & \sigma \\
         -\sigma & 0\\
       \end{array}
     \right]v.
\end{align*}
The regulated error is defined as $e =y -v_{1}$. 
Assume that $b(v,w) =-4+0.05v_1$ and $\sigma=\pi/12$. The system can be shown to satisfy all the assumptions in \cite{liu2009parameter}. The controllers \eqref{escun} and \eqref{escun2} are designed with $k=1.5$, $\alpha=4$,
$m=\col(24,50,35,10)$ and $\Theta=10$.
The simulation is conducted with the following initial conditions: $v(0)=\col(9,1)$, $\eta(0)=\col(0.1589;0.0622;0.1057;0.0331)$, $\pi(0)=0$ and $\col\left(x_{1i},x_{2i},y \right)=(0,0,0)$. All other initial conditions in the controller are set to zero. The uncertain parameter is $w =\textnormal{col}\left(9, 1\right)$.

The extremum seeking controllers \eqref{escun}, \eqref{escun2} are compared to Nussbaum gain schemes for system \eqref{example-put-feedback}. The closed-loop system with the Nussbaum gain is given by
\begin{subequations}\label{nussbaum}\begin{align}
{u}&={\myr \mathcal{N}(k_n)}\rho(e)e+\chi_s(\eta, \vartheta),\\
\dot{k}_n&=\rho(e)e^2\\
\dot{\eta}&=M\eta+N\pi,\\
  \dot{\pi}&=-\pi+u,\\
  \dot{\vartheta}&=-\Theta\eta[\eta^{\!\top}\vartheta-\pi],
\end{align}\end{subequations}
where ${\myr\mathcal{N}(k_n)}=k_n^2\cos(k_{n})$ is a type
of Nussbaum function as described in
\cite{nussbaum1983some} and \cite{liu2008global}.

The resulting closed-loop trajectories using the Nussbaum gain and the ESC control system are shown in Figures \ref{fige}, \ref{figy}, \ref{fige2} and \ref{figy2}. The ESC control system is tested at varying frequencies. Figure \ref{fige} shows the trajectories of $e =y -v_{1}$  for the Nussbaum gain technique \eqref{nussbaum} and the ESC approach \eqref{escun} with $\rho(s)=s^2+20$.
Figure \ref{figy} shows the trajectory of $y$.
Figure \ref{fige2} shows the trajectory of $e=y -v_{1}$ with $\rho(s)=s+20$. The trajectory of $y$ is shown in Figure \ref{figy2}.
From Figures \ref{fige} and \ref{fige2}, the large overshoot phenomenon can be observed, even when the suitable equilibrium has been reached, when the Nussbaum gain technique \eqref{nussbaum} is used. These large deviations are not observed for the extremum-seeking control approach. It is important to note that the extremum-seeking control approach can require larger frequencies, which may be undesirable in some applications. To overcome this, one can choose a smaller value of the parameter $\alpha$ to offset the need for larger frequencies. In Figures \ref{figalpha} and \ref{figalpha2}, the impact of choosing smaller values for $\alpha$ is demonstrated.  It is seen that more precise convergence to the equilibrium can be achieved by reducing $\alpha$ at a fixed frequency.

\section{Conclusion}\label{section6}
This paper has studied the practical robust output regulation problem of a class of nonlinear systems subject to unknown control directions and an uncertain exosystem.
By employing an extremum-seeking control approach, we proposed control laws that handle the robust practical output regulation problem subject to unknown control directions with time-varying coefficients.
An analysis of robust non-adaptive stabilization problems is performed for an augmented system with iISS inverse dynamics.
%
%
The stability of the non-adaptive output regulation design via a Lie bracket averaging technique is demonstrated. A uniform ultimate boundedness of the closed-loop signals is guaranteed.
%
%
%
%
{\myr It is shown that the proposed method can address an output regulation problem with unknown control directions and an uncertain exosystem without utilizing the Nussbaum-type gain technique, thereby strengthening the leading approach of the existing framework proposed in \cite{liu2006global,guo2016cooperative}.}
%

\bibliographystyle{ifacconf}
\bibliography{myref}
\end{sloppypar}
\end{document}